\documentclass[reqno,12pt]{amsart}%
\usepackage[dvips]{graphicx}
\usepackage{mathrsfs}
\usepackage{color}
\usepackage{amsmath}
\usepackage{amsfonts}
\usepackage{amssymb}%

\newcommand{\dotW}{\dot{W}}

\newcommand{\cA}{\mathcal{A}}

\newcommand{\cU}{\mathcal{U}}
\newcommand{\cJ}{\mathcal{J}}
\newcommand{\cR}{\mathcal{R}}

\newcommand{\cL}{\mathcal{L}}
\newcommand{\cD}{\mathcal{D}}
\newcommand{\cS}{\mathcal{S}}

\newcommand{\bE}{\mathbb{E}}
\newcommand{\bF}{\mathbb{F}}
\newcommand{\bN}{\mathbb{N}}
\newcommand{\bR}{\mathbb{R}}

\newcommand{\bD}{\mathbf{D}}

\newcommand{\mfu}{\mathfrak{u}}
\newcommand{\mfp}{\mathfrak{p}}
\newcommand{\mfq}{\mathfrak{q}}
\newcommand{\mfr}{\mathfrak{r}}
\newcommand{\mfc}{\mathfrak{c}}

\newcommand{\bs}[1]{{\boldsymbol{#1}}}
\newcommand{\ba}{\boldsymbol{\alpha}}
\newcommand{\bbt}{\boldsymbol{\beta}}
\newcommand{\bg}{\boldsymbol{\gamma}}
\newcommand{\bep}{\boldsymbol{\varepsilon}}

\newcommand{\bdel}{\boldsymbol{\delta}}

\def\Hep{{\mathrm{H}}}

\newtheorem{theorem}{Theorem}
\theoremstyle{plain}
\newtheorem{corollary}[theorem]{Corollary}
\newtheorem{definition}[theorem]{Definition}
\newtheorem{proposition}[theorem]{Proposition}
\newtheorem{example}[theorem]{Example}

\newtheorem{remark}[theorem]{Remark}

\numberwithin{equation}{section}
\numberwithin{theorem}{section}

\begin{document}

\title[Generalized Malliavin Calculus]
{A Note on Generalized Malliavin Calculus}
\author{S. V. Lototsky,  B. L. Rozovskii, and D. Sele\v{s}i}

\thanks{S. V. Lototsky acknowledges support from NSF Grant DMS-0803378.
B. L. Rozovskii acknowledges support from NSF Grant DMS-0604863, ARO
Grant W911NF-07-1-0044, and AFOSR Grant 5-21024 (inter). D. Sele\v
si acknowledges support from Project No. 144016 (Functional analysis
methods, ODEs and PDEs with singularities) financed by the Ministry
of Science, Republic of Serbia.}

\begin{abstract}
The Malliavin derivative, divergence operator, and the Ornstein-Uhlenbeck operator
 are extended from the traditional Gaussian setting to generalized processes from the higher-order chaos spaces.
\end{abstract}

\maketitle

\section{Introduction}
\label{intro}

Stochastic integration started with the construction of integrals
with respect to the Wiener process \cite{Ito-orig} and then extended to a much
larger class of processes \cite{Protter}. The Wiener process was also in the core
of the early development of Malliavin calculus \cite{Malliav-orig},
 but generalizations
so far have not been nearly as sweeping as in the theory of
stochastic integration. Currently, the \textit{driving random
source} in Malliavin calculus is an isonormal Gaussian process on a
separable Hilbert space \cite{Malliavin-b1, Nualart}. This process, similar to the standard
Wiener process, is in effect a \textit{linear} combination of a
countable collection
$\boldsymbol{\xi}:=\left\{ \xi_{i}\right\}
_{i\geq1}$ of independent standard Gaussian random variables.

A natural question to ask is whether one can extend Malliavin
calculus to \textit{nonlinear} functionals of isonormal Gaussian
process as the driving random source, and still enjoy all the
benefits of the Gaussian setting. Natural candidates for this role
are elements of the Hilbert space of square integrable functionals
of the isonormal Gaussian process. This space is often referred to
as Wiener Chaos space.

In this paper we extend the main operators of Malliavin calculus to
the space of generalized random elements $\sum_{\left\vert
\ba\right\vert
<\infty}f_{\ba}\xi_{\ba},$ where $\left\{
\xi_{\boldsymbol{\alpha}},\left\vert \ba\right\vert
<\infty \right\}  $ is the Cameron-Martin basis in the Wiener Chaos
space, $\ba$ is a multiindex and $f_{\ba}$
belong to a certain Hilbert space $X$ (for detail see Section 2). To
cover some emerging applications, we allow formal linear
combinations with infinite variance, that is $\sum_{\left\vert
\ba\right\vert <\infty}\left\Vert
f_{\ba}\right\Vert _{X}^{2}=\infty.$

Looking for solutions that are generalized random elements is quite
reasonable: after all, the Gaussian white noise that often drives
the equation of interest is itself a generalized random element. Our
interest in this subject was prompted by some recent and not so
recent developments in the stochastic partial differential equations
(SPDEs). These developments indicate that large classes of solutions
of linear and nonlinear SPDEs driven by Gaussian sources are
generalized random elements.

One example is the heat equation driven by multiplicative space-time
white noise $\dot{W}(t,x)$ with dimension of $x$ two or higher \cite{NR}:
\begin{equation}
u_{t}=\boldsymbol{\Delta}u+u\dot{W},\ u(0,x)=e^{-|x|^2}. \label{in-ex1}%
\end{equation}
Examples in one space dimension also exist:
a stochastic parabolic equation violating the parabolicity condition \cite{LR_AP}:
\begin{equation}
du=u_{xx}dt+\sigma u_{x}dw(t),\ \sigma^{2}>2, \ u(0,x)=e^{-x^2},\label{in-ex2}%
\end{equation}
or a stochastic parabolic equation of full second order \cite{LR-fs}:
\begin{equation}
du=u_{xx}dt+u_{xx}dw(t),\ u(0,x)=e^{-x^2}. \label{in-ex3}%
\end{equation}
In all three examples, $\bE\|u(t,\cdot)\|_X^2=\sum_{|\ba|<\infty} \|u_{\ba}(t,\cdot)\|_X^2=\infty$
for all $t>0$ and all typical function spaces $X$, such as Sobolev spaces.

As a different example, consider equation
\begin{equation}
-\left(  a\left(  x\right)  u_{x}(x)\right)  _{x}=f(x),\,x\in\left(
0,1\right)  ,\,u\left(  0\right)  =u\left(  1\right)  =0, \label{eleq}%
\end{equation}
with $a(x)=\bar{a}(x)+\epsilon(x),$ where $\bar{a}(x)$ is non-random
and $\epsilon(x)=\sum_{k\geq1}\sigma_{k}(x)\xi_{k}$ is a Gaussian
noise term; $\sum_{k\geq 1} \sup_{x}\sigma_k^2(x)<\infty$. Recently, this equation was investigated in the context
of uncertainty quantification for mathematical and computational
models  \cite{WRK}. As problem (\ref{eleq}) is
ill posed, one could modify it as follows:%
\begin{equation}%
\begin{array}
[c]{c}%
-\left(  \bar{a}\left(  x\right)  v_{x}(x)\right)  _{x}+\left(
\delta
_{\epsilon(x)}\left(  v_{x}\left(  x\right)  \right)  \right)  _{x}=f(x),\\
\,x\in\left(  0,1\right)  ,\,v\left(  0\right)  =v\left(  1\right)
=0,
\end{array}
\label{spde}%
\end{equation}
where $\delta_{\boldsymbol{\epsilon}(x)}$ stands for Malliavin
divergence operator (Skorokhod integral) with respect to Gaussian
noise $\epsilon(x)$. In contrast to (\ref{eleq}), equation
(\ref{spde}) is well posed and uniquely solvable, and similar to equations
\eqref{in-ex1}--\eqref{in-ex3},
$\mathbb{E}\left\Vert v\right\Vert_X^{2}=\infty$ for all traditional spaces $X$ of  functions on $(0,1)$.  Technically, the above modification of problem
(\ref{eleq}) amounts to replacement of products of random elements
by stochastic convolutions, such as  Wick products\cite{HKPS, HOUZ, Wick-orig}. In the literature on quantum physics,
procedures of this type are often called \textit{stochastic
quantization} \cite{GJ, Simon-qft}.
Equations subjected to the stochastic quantization procedure are usually
referred to as {\em quantized}.

We remark that the replacement of equation (\ref{eleq}) by equation
(\ref{spde}) also mimics the idea of It\^{o} \cite{Ito-orig} of replacing the
singular equation
\[
\dot{x}\left(  t\right)  =a\left(  x\left(  t\right)  \right)
+\sigma\left( x\left(  t\right)  \right)  \dot{w}(t)
\]
by the well posed stochastic differential equation
\[
x\left(  t\right)  =x_{0}+\int_{0}^{t}a\left(  x\left(  s\right)
\right) ds+\int_{0}^{t}\sigma\left(  x\left(  s\right)  \right)
dw(s).
\]
Because there is no natural filtration associated with elliptic
equation (\ref{spde}), the It\^{o} integral has to be replaced
by the Skorokhod integral.

Equation (\ref{spde}) differs from (\ref{eleq}) quite drastically.
While both equations are stochastic perturbations to the solution of
the deterministic equation
\begin{equation}
-\left(  \bar{a}\left(  x\right)  \bar{v}_{x}(x)\right)
_{x}=f(x),\,x\in \left(  0,1\right)  ,\,\bar{v}\left(  0\right)
=\bar{v}\left(  1\right),
=0\label{eq:det}%
\end{equation}
only the solution to the quantized equation (\ref{spde}) is an
\textit{unbiased}
perturbation of the solution of equation (\ref{eq:det}) in that $\mathbb{E}%
v\left(  x\right)  =\bar{v}\left(  x\right)  $ is a solution of
equation (\ref{eq:det}); the solution of equation
(\ref{eleq}), even if existed, would not enjoy this property.

Two other  examples of stochastic quantization are currently under investigation:
 randomly forced Burgers equation \cite{KaligLot-QBE} and Navier-Stokes equation
\cite{MikRoz-QNSE}. Let us consider Burgers equation
\begin{equation}
u_{t}=u_{xx}+uu_{x}+e^{-x^{2}}\xi, \ t>0,\ x \in \bR, \label{eq:burgers}%
\end{equation}
with a deterministic initial condition, where $\xi$ is a standard Gaussian random variable. The stochastic
quantization of this equation is
\begin{equation}
v_{t}=v_{xx}+\delta_{v}\left(  v_{x}\right)  +e^{-x^{2}}\xi,
\label{eq:Qburgers}%
\end{equation}
where $\delta_{v}\left(  v_{x}\right)  $ is the Malliavin divergence
operator (Skorokhod integral) of $v_{x}$ with respect to the
solution $v$ of (\ref{eq:Qburgers}) (and $v$ is not a Gaussian
process). It is shown \cite{KaligLot-QBE} that $\bar{v}\left(  t,x\right)
:=\mathbb{E}v\left(  t,x\right)$, with a suitable interpretation of the
 expectation, is the solution of the
deterministic Burgers
\[
\bar{v}_{t}\left(  t,x\right)  =\bar{v}_{xx}\left(  t,x\right)
+\bar {v}\left(  t,x\right)  \bar{v}_{x}\left(  t,x\right)  .
\]
Thus, as in the linear example \eqref{eleq}, the quantized version of
stochastic Burgers equation (\ref{eq:Qburgers}) is an unbiased
perturbation of the deterministic Burgers equation with the same
initial condition. The  standard stochastic Burgers equation
does not have this convenient property.
Similar effect holds for quantized Navier-Stokes equation  \cite{MikRoz-QNSE}.

We emphasize that, in all examples we have discussed, the variance of a
generalized
random element $u$ is infinite and is given by the diverging sum $\sum_{\boldsymbol{\alpha}%
}\left\Vert u_{\boldsymbol{\alpha}}\right\Vert _{X}^{2}=\infty.$
However, the rate of divergence  differs substantially from case
to case. To study this rate of divergence, we introduce a rescaling, or weighting,
operator  $\mathfrak{R}$
defined by $\mathfrak{R}\xi_{\boldsymbol{\alpha}}=r_{\boldsymbol{\alpha}}%
\xi_{\boldsymbol{\alpha}},$  where \textit{weights}
$r_{\boldsymbol{\alpha}}$ are positive numbers selected in such a
way that the weighted sum $\sum_{\ba}r_{\ba}^{2}\left\Vert
u_{\ba}\right\Vert _{X}^{2}$ becomes finite. Of course, the choice of
$r_{\ba}$ is not unique and
 depends on the specifics of the problem, for example on the
type of the stochastic PDE in question. A special case of this
 rescaling procedure originates in
quantum physics and is related to \textit{second
quantization} \cite{Simon-qft}.

Quantum physics has brought about a number of important precursors
to Malliavin calculus. For example, \textit{creation} and
\textit{annihilation} operators correspond to Malliavin divergence
and derivative operators, respectively, with respect to a single Gaussian random
variable. The original definition of Wick product \cite{Wick-orig}
is not related to the Malliavin divergence operator or Skorokhod
integral but remarkably these notions coincide in some situations.
In fact, standard Wick product could be interpreted as Skorokhod
integral with respect to square integrable processes generated by
Gaussian white noise, while the classic Malliavin divergence
operator integrates only with respect to isonormal Gaussian process.
In Section \ref{sec:gen}, we demonstrate that Malliavin divergence
operator can be extended to the setting where both the integrand
and the integrator are generalized random elements in a Hilbert
space, although  we did not try to extend Wick product in a similar way.

In this paper, we restrict ourselves to the basic study of the three
main operators in the Malliavin calculus: the derivative operator
${\mathbf{D}}$, the divergence operator $\boldsymbol{\delta}$ and
the Ornstein-Uhlenbeck operator
${\mathcal{L}}=\boldsymbol{\delta}\circ{\mathbf{D}}$. We present
constructions of ${\mathbf{D}}_{u}(v)$,
$\boldsymbol{\delta}_{u}(f)$, and ${\mathcal{L}}_{u}(v)$ when
$u,v,f$ are Hilbert space-valued generalized random elements.
Section \ref{sec:rev} reviews the main constructions of the
Malliavin calculus in the form suitable for generalizations. Section
\ref{sec:gen} presents the definitions of the Malliavin derivative,
Skorokhod integral, and Ornstein-Uhlenbeck operator in the most
general setting of weighted chaos spaces. Section \ref{sec:ex}
presents a more detailed analysis of the operators on some special
classes of spaces.

To illustrate some of the  main results, let us consider the
one-dimensional setting. Let $\xi$ be a standard normal random
variable and define
\[
\xi_{(n)}=\frac{{\mathrm{H}}_{n}(\xi)}{\sqrt{n!}},\ n\geq0,
\]
where ${\mathrm{H}}_{n}$ is $n$th Hermite polynomial. If $f$ is a
square-integrable functional of $\xi$, then
\[
f=\sum_{n\geq0}\mathbb{E}\left(  f\xi_{(n)}\right)  \xi_{(n)}.
\]
The space of square-integrable functionals of $\xi$ can thus be
identified with $\ell_{2}$:
\[
\{f_{n},\ n\geq0:\sum_{n\geq0}f_{n}^{2}<\infty\}.
\]
We define a generalized random functional $f$ of $\xi$ as a
collection of numbers $\{f_{n},\ n\geq0\}$ without any restrictions
on $f_{n}$ and a formal representation
\[
f=\sum_{k\geq0}f_{n}\xi_{(n)}.
\]
Let $u,f,v$ be generalized functionals of $\xi$ and let $p,q,r$ be
positive real numbers such that
\[
\frac{1}{p}+\frac{1}{q}=\frac{1}{r},
\]
We show in the paper that if
\[
\sum_{n\geq0}p^{n}u_{n}^{2}<\infty,\ \sum_{n\geq0}\frac{v_{n}^{2}}{r^{n}%
}<\infty,
\]
then ${\mathbf{D}}_{u}(v)$ is a generalized functional of $\xi$ such
that
\[
\big({\mathbf{D}}_{u}(v)\big)_{n}=\sum_{k=0}^{\infty}\left(  \frac
{(n+k)!}{n!k!}\right)  ^{1/2}v_{n+k}u_{k},
\]
and
\[
\sum_{n\geq1}\frac{\big({\mathbf{D}}_{u}(v)\big)_{n}^{2}}{q^{n}}<\infty,
\]
Similarly, if
\[
\sum_{n\geq0}p^{n}u_{n}^{2}<\infty,\
\sum_{n\geq0}q^{n}{f_{n}^{2}}<\infty,
\]
then $\boldsymbol{\delta}_{u}(f)$ is a generalized functional of
$\xi$ such that $\boldsymbol{\delta}_{u}(f)=u\diamond f$, where
operator $\diamond$ stands for the Wick product,
\[
\big(\boldsymbol{\delta}_{u}(f)\big)_{n}=\sum_{k=0}^{n}\left(  \frac
{n!}{k!(n-k)!}\right)  ^{1/2}f_{k}u_{n-k},
\]
and
\[
\sum_{n\geq1}r^{n}{\big(\boldsymbol{\delta}_{u}(f)\big)_{n}^{2}}<\infty.
\]
Finally, if $p,q,r$ are positive real numbers such that
\[
\left(  \frac{1}{r}-\frac{1}{p}\right)  \left(  q-\frac{1}{p}\right)
=1
\]
(for example, $p=1,\ q=2,\ r=1/2$) and
\[
\sum_{n\geq0}p^{n}u_{n}^{2}<\infty,\
\sum_{n\geq0}q^{n}{v_{n}^{2}}<\infty,
\]
then ${\mathcal{L}}_{u}(v)$ is a generalized functional of $\xi$
such that
\[
\big({\mathcal{L}}_{u}(v)\big)_{n}=\sum_{k=0}^{n}\sum_{m=0}^{\infty}%
v_{k+m}u_{n-k}u_{m}%
\]
and
\[
\sum_{n\geq0}r^{n}\big({\mathcal{L}}_{u}(v)\big)_{n}^{2}<\infty.
\]

\section{Review of the traditional Malliavin Calculus}
\label{sec:rev}

The starting point in the development of Malliavin calculus
is the {\em isonormal Gaussian process} (also known as Gaussian white noise) $\dot{W}$: a  Gaussian system $\{\dot{W}(u),\ u\in \cU\}$ indexed by a separable Hilbert space $\cU$ and  such that
$\bE \dot{W}(u)=0$, $\bE\big(\dot{W}(u)\, \dot{W}(v)\big)=(u,v)_{\cU}$.
The objective of this section is to outline a different but equivalent construction.

Let $\mathbb{F}:=\left(  \Omega,\mathcal{F},\mathbb{P}\right)$
 be a probability
space, where $\mathcal{F}$ is the $\sigma$-algebra generated by
 a collection of independent standard Gaussian random variables
 $\left\{  \xi_{i}\right\}  _{i\geq1}$.
Given a real separable Hilbert space $X$, we
denote by $L_{2}({\mathbb{F}};X)$ the Hilbert space of square-integrable
${\mathcal{F}}$-measurable $X$-valued random elements $f$.
When $X={\mathbb{R}}$, we often write
$L_{2}({\mathbb{F}})$ instead of $L_{2}({\mathbb{F}};{\mathbb{R}})$.
Finally, we fix a real separable Hilbert space ${\mathcal{U}}$ with an orthonormal basis
$\mathfrak{U}=\{\mfu_k,\ k\geq 1\}$.
\begin{definition}
\label{def:well} A {\tt  Gaussian white noise} $\dotW$ on
$\mathcal{U}$ is a formal series
\begin{equation}
\label{wn}
\dot{W}=\sum_{k\geq1}\xi_{k}
{\mathfrak{u}}_{k}.
\end{equation}
\end{definition}
Given an isonormal Gaussian process $\dotW$ and an orthonormal basis $\mathfrak{U}$ in $\cU$,
representation \eqref{wn} follows with $\xi_k=\dot{W}(\mathfrak{u}_k)$.
Conversely, \eqref{wn} defines an isonormal Gaussian process on $\cU$ by
$$
\dot{W}(u)=\sum_{k\geq 1} (u,\mathfrak{u}_k)_{\cU}\, \xi_k.
$$
To proceed, we need to review several definitions
 related to {\em multi-indices}. Let ${\mathcal{J}}$ be the
collection of multi-indices
$\boldsymbol{\alpha}=(\alpha_{1},\alpha_{2},\ldots)$ such that
$\alpha_{k}\in\{0,1,2,\ldots\}$ and $\sum_{k\geq1}\alpha_{k}<\infty$. For
$\boldsymbol{\alpha},\boldsymbol{\beta}\in{\mathcal{J}}$,
we define
\begin{equation*}
\boldsymbol{\alpha}+\boldsymbol{\beta}
=(\alpha_{1}+\beta_{1},\alpha_{2}+\beta_{2},\ldots), \  \
|\boldsymbol{\alpha}|=\sum_{k\geq1}\alpha_{k}, \ \
\boldsymbol{\alpha}!=\prod_{k\geq1}\alpha_{k}!.
\end{equation*}
 By definition, $\boldsymbol{\alpha}>0$ if
$|\boldsymbol{\alpha}|>0$ and
${\boldsymbol{\beta}}\leq\boldsymbol{\alpha}$
 if $\ {\beta}_{k}\leq {\alpha}_{k}$ for all
$k\geq1.$ If ${\boldsymbol{\beta}}\leq\boldsymbol{\alpha}$,
then
$$
\boldsymbol{\alpha}-{\boldsymbol{\beta}}
=(\alpha_{1}-\beta_{1},\alpha_{2}-\beta_{2},\ldots).
$$
Similar to the convention for the usual binomial coefficients,
$$
\binom{\boldsymbol{\alpha}}{\bbt}
=
\begin{cases}
\displaystyle \frac{\ba!}{(\ba-\bbt)!\bbt!},&\ {\rm if } \ \bbt\leq \ba,\\
0, & {\rm otherwise.}
\end{cases}
$$
We use the following notation for the special
multi-indices:
\begin{enumerate}
\item ${\boldsymbol{(0)}}$ is the multi-index with all zero
entries: ${\boldsymbol{(0)}}_{k}=0$ for all $k$;
\item  $\bep(i)$ is the multi-index of length $1$
and with the single
non-zero entry at position $i$: i.e.
$\bep(i)_k=1$ if $k=i$ and $\bep(i)_k=0$
if $k\neq i.$ We also use convention
$\bep(0)={\boldsymbol{(0)}}$.
\end{enumerate}
Given a sequence of positive numbers $\mfq=(q_1, q_2,\ldots)$
 and a real number $\ell$, we define the  sequence
$\mfq^{\ell\ba},\ \ba\in \cJ$, by
$$
\mfq^{\ba}=\prod_k q_k^{\ell\alpha_k}.
$$
In particular,
$$
 (2{\mathbb{N}})^{\ell\alpha}=\prod_{k\geq
1}(2k)^{\ell\alpha_{k}}.
$$
Next, we recall the construction of an orthonormal basis in
$L_{2}({\mathbb{F}};X)$.
Define the collection of random variables
$$
\Xi=\big\{ {\xi}_{\boldsymbol{\alpha}},\
\boldsymbol{\alpha}\in{\mathcal{J}}\big \}
$$
as follows:
$$
\xi_{\boldsymbol{\alpha}}=
\prod_{k\geq1}\frac{ \Hep_{\alpha_{k}}(\xi_{k})}{\sqrt{\alpha_{k}!}},
$$
 where $\Hep_{n}$ is the Hermite polynomial of order $n$:
$$
 \Hep_{n}(t)=(-1)^ne^{t^{2}/2}
\frac{d^{n}}{dt^{n}}e^{-t^{2}/2}.
$$
Sometimes it is convenient to work with unnormalized basis elements
$\Hep_{\ba}$, defined by
\begin{equation}
\label{Hep-a}
\Hep_{\ba}=\sqrt{\ba!} \, \xi_{\ba}.
\end{equation}
\begin{theorem}[Cameron-Martin \protect{\cite{CM}}]
\label{th:CM}
The set $\Xi$ is an orthonormal
basis in $L_{2}(\mathbb{F};X)$: if $v\in L_{2}(\mathbb{F};X)$ and
$v_{\boldsymbol{\alpha}}
=\mathbb{E}\big(v\,\xi_{\boldsymbol{\alpha}}\big)$, then
$v=\sum_{\ba\in\mathcal{J}}v_{\alpha}\xi_{\alpha}$
and $\mathbb{E}
\|v\|_X^{2}=\sum_{\alpha\in\mathcal{J}}\|v_{\alpha}\|_X^{2}$.
\end{theorem}
If the space $\cU$ is finite-dimensional, then the multi-indices are
restricted to the the set
$$
\cJ_n=\{\ba\in \cJ: \alpha_k=0,\, k>n\}.
$$
The three main operators of the Malliavin calculus are
\begin{enumerate}
\item The   {\tt (Malliavin) derivative}
 $\mathbf{D}_{\dot{W}}$;
 \item    The {\tt divergence operator} $\bdel_{\dot{W}}$, also
 known as the {\tt Skorokhod integral};
\item  The {\tt Ornstein-Uhlenbeck operator} $\cL_{\dot{W}}= \bdel_{\dot{W}} \, \bD_{\dot{W}}$.
\end{enumerate}
For reader's convenience, we summarize the main properties of
 $\mathbf{D}_{\dot{W}}$ and ${\boldsymbol{\delta}}_{\dot{W}}$;
 all the details are in  \cite[Chapter 1]{Nualart}.
 \begin{enumerate}
 \item  $\mathbf{D}_{\dot{W}}$ is a closed unbounded linear operator
 from $L_{2}({\mathbb{F}};X)$ to $L_{2}({\mathbb{F}};X\otimes \cU)$;
 the domain of $\mathbf{D}_{\dot{W}}$ is denoted by  ${\mathbb{D}}^{1,2}(\bF;X)$;
 \item If $v=F\big(\dot{W}(h_1),\ldots,\dot{W}(h_n)\big)$ for a
 polynomial $F=F(x_1, \ldots, x_n)$ and $h_1,\ldots, h_n\in X$, then
 \begin{equation}
 \label{MD-def}
 \mathbf{D}_{\dot{W}}(v)=\sum_{k=1}^n \frac{\partial F}{\partial x_k}\big(\dot{W}(h_1),\ldots,\dot{W}(h_n)\big)\,h_k.
 \end{equation}
 \item  ${\boldsymbol{\delta}}_{\dot{W}}$ is the adjoint of $\mathbf{D}_{\dot{W}}$ and is a closed unbounded linear operator
 from $L_{2}({\mathbb{F}};X\otimes \cU)$ to $L_{2}({\mathbb{F}};X)$
 such that
 \begin{equation}
\mathbb{E}\big(\varphi{\boldsymbol{\delta}}_{\dot{W}}(f)\big)
= {\mathbb{E}}\big(f,{\mathbf{D}}_{\dot{W}}(\varphi)\big)_{{\mathcal{U}}}
 \label{eq:SkI}       %
\end{equation}
for all $\varphi\in{\mathbb{D}}^{1,2}({\mathbb{F};\bR})$ and
$f\in{\mathbb{D}}^{1,2}({\mathbb{F};X\otimes \cU})$. Equivalently,
\begin{equation}
\big(v,{\boldsymbol{\delta}}_{\dot{W}}(f)\big)_{L_2(\bF;X)}
= \big(f,{\mathbf{D}}_{\dot{W}}(v)\big)_{L_2(\bF;X\otimes \cU)}
 \label{eq:SkI-ip}       %
\end{equation}
for all $v\in{\mathbb{D}}^{1,2}({\mathbb{F};X})$ and
$f\in{\mathbb{D}}^{1,2}({\mathbb{F};X\otimes \cU})$.
\end{enumerate}
We will need representations of the operators ${\mathbf{D}}_{\dot{W}}$,
${\boldsymbol{\delta}}_{\dot{W}}$, and $\cL_{\dot{W}}$ in the
basis $\Xi$.
\begin{theorem}
\label{prop:MC-orig}

(1) If  $v\in L_2(\bF;X)$ and
\begin{equation}
\label{MD-cond}
\sum_{\ba \in \cJ} |\ba|\, \|v_{\ba}\|_X^2 <\infty,
\end{equation}
then $\mathbf{D}_{\dot{W}}(v)\in L_2(\bF;X\otimes \cU)$ and
\begin{equation}
\label{MD-wn}
\mathbf{D}_{\dot{W}}(v)=\sum_{\ba\in \cJ} \sum_{k\geq 1}
\sqrt{\alpha_{k}}\,\xi_{\ba-\bep(k)}\,
 v_{\ba}\otimes{\mathfrak{u}}_{k}.
\end{equation}
(2) If
$$
f=\sum_{\ba\in \cJ,\,k\geq 1} f_{k,\ba}\otimes \mfu_k\,
\xi_{\ba},
$$
and
\begin{equation}
\label{DO-cond}
\sum_{\ba\in \cJ,\,k\geq 1} |\ba|\|f_{k,\ba}\|^2_X<\infty,
\end{equation}
then
$\bdel_{\dot{W}}(f)\in L_2(\mathbb{F};X)$ and
\begin{equation}
\label{DO-wn}
\bdel_{\dot{W}}(f)=\sum_{\ba\in \cJ} \left(\sum_{k\geq 1}
\sqrt{\alpha_k} f_{k,\ba-\bep(k)}\right)\xi_{\ba}.
\end{equation}
(3) If  $v\in L_2(\bF;X)$ and
\begin{equation}
\label{OU-cond}
\sum_{\ba \in \cJ} |\ba|^2\, \|v_{\ba}\|_X^2 <\infty,
\end{equation}
then $\cL_{\dot{W}}(v) \in  L_2(\bF;X) $ and
\begin{equation}
\label{OU-wn}
\cL_{\dot{W}}(v)=\sum_{\ba\in \cJ} |\ba|v_{\ba}\, \xi_{\ba}.
\end{equation}
\end{theorem}

\begin{proof} Linearity of the operators implies that, in each case, it is enough to
find the image of $\xi_{\ba}$.

(1)  Using \eqref{MD-def} and properties of the Hermite polynomials,
 for every $\ba\in{\mathcal{J}},$
\begin{equation}
{\mathbf{D}}_{\dot{W}}(\xi_{\ba})
=\sum_{k\geq 1}\sqrt{\alpha_{k}}\,\xi_{\ba-\bep(k)}\,{\mathfrak{u}}_{k}.
\label{eq:Dxi}
\end{equation}
Orthonormality of $\{\xi_{\ba},\ \ba\in \cJ\}$
and $\{\mfu_k,\ k\geq 1\}$ then implies
\begin{equation*}
\begin{split}
\bE \|\mathbf{D}_{\dot{W}}(v)\|_{X\otimes \cU}^2 &=
\sum_{\ba,\bbt,k,n} \sqrt{\alpha_k\, \beta_n}\,
\bE\big(\xi_{\ba-\bep(k)}\xi_{\bbt-\bep(n)}\big)
(v_{\ba},v_{\bbt})_X\, (\mfu_k,\mfu_n)_{\cU}\\
& = \sum_{\ba,k}\alpha_k\|v_{\ba}\|_X^2 =
\sum_{\ba} |\ba|\, \|v_{\ba}\|_X^2.
\end{split}
\end{equation*}
(2) By \eqref{eq:Dxi} and \eqref{eq:SkI},
for every $\xi_{\alpha}\in \Xi$, $h\in X$,
and $\mfu_{k}\in \mathfrak{U}$,
\begin{equation}
{\boldsymbol{\delta}}_{\dot{W}}(h\otimes{\mathfrak{u}}_{k}\,\xi_{\alpha})
=h\,\sqrt{\alpha_{k}+1}\,\xi_{\alpha+\varepsilon_{k}}.
 \label{eq:Ixi}       %
\end{equation}
(3) By \eqref{eq:Dxi} and \eqref{eq:Ixi},
for every $\xi_{\ba}$,
\begin{equation}
\label{OU-basis}
\cL_{\dotW}(\xi_{\ba})=
\bdel_{\dot{W}} \Big(\bD_{\dot{W}} (\xi_{\ba}) \Big) =
|\ba| \xi_{\ba}.
\end{equation}
\end{proof}

\begin{remark}
\label{rem:main} {\rm
Here is an important technical difference between the derivative and the
divergence operators:
\begin{itemize}
\item  For the operator $\mathbf{D}_{\dot{W}}$,
\begin{equation}
\label{aux1}
\big(\mathbf{D}_{\dot{W}}(v)\big)_{\ba}=
\sum_{k\geq 1}
\sqrt{\alpha_k+1}\, v_{\ba+\bep(k)} \otimes \mfu_k;
\end{equation}
in general, the sum on the right-hand side contains infinitely many terms and will
diverge without additional conditions on $v$, such as \eqref{MD-cond}.
 \item For the operator $\bdel_{\dot{W}}$,
\begin{equation}
\label{DO-aux1}
\big(\bdel_{\dot{W}}(f)\big)_{\ba}=
\sum_{k\geq 1} \sqrt{\alpha_k} f_{k,\ba-\bep(k)};
\end{equation}
 the sum on the right-hand side always contains
finitely many terms, because only finitely many of
$\alpha_k$ are not equal to zero. Thus, for {\em fixed} $\ba$,
$\big(\bdel_{\dot{W}}(f)\big)_{\ba}$ is defined without any additional
conditions on $f$.
\end{itemize}
}
\end{remark}

\section{Generalizations to weighted chaos spaces}
\label{sec:gen}

Recall that  $\dot{W}$, as defined by \eqref{wn}, is not a $\cU$-valued random element,
 but  a {\em generalized}  random element on $\cU$:
\begin{equation}
\label{wn1}
\dot{W}(h)=\sum_{k\geq 1} (h,\mfu_k)_{\cU} \, \xi_k,
\end{equation}
where the series on the right-hand side converges with probability
one for every $h\in \cU$.
The objective of this section is to find similar interpretations
of the series in \eqref{MD-wn},
\eqref{DO-wn}, and \eqref{OU-wn} if the corresponding conditions \eqref{MD-cond},
\eqref{DO-cond}, \eqref{OU-cond} fail. Along the way,
it also becomes natural to allow other generalized random
elements to replace $\dot{W}$.

We start with the construction of {\em weighted chaos spaces}.
Let ${\mathcal{R}}$ be a bounded linear operator on $L_{2}({\mathbb{F}})$
defined by ${\mathcal{R}}\xi_{\ba}=r_{\ba}\xi_{\ba}$ for every
$\ba\in{\mathcal{J}}$, where the \emph{weights} $\{r_{\ba},\ \ba
\in{\mathcal{J}}\}$ are positive numbers.

Given a Hilbert space $X$,
we extend ${\mathcal{R}}$ to an operator on $L_{2}({\mathbb{F}};X)$ by
defining ${\mathcal{R}}f$ as the unique element of $L_{2}({\mathbb{F}};X)$ so
that, for all $g\in L_{2}(\mathbb{F};X)$,
\[
{\mathbb{E}}({\mathcal{R}}f,g)_{X}
= \sum_{\ba\in{\mathcal{J}}}r_{\ba
}{\mathbb{E}}\big((f,g)_{X}\xi_{\ba}\big).
\]
Denote by ${\mathcal{R}}L_{2}({\mathbb{F}};X)$ the closure of $L_{2}       %
({\mathbb{F}};X)$ with respect to the norm
\[
\Vert f\Vert_{{\mathcal{R}}L_{2}({\mathbb{F}};X)}^{2}
:=\Vert{\mathcal{R}}f\Vert_{L_{2}({\mathbb{F}};X)}^{2}.
\]

In what follows, we will identify the operator ${\mathcal{R}}$ with the
corresponding collection $(r_{\ba},\ \ba\in{\mathcal{J}})$. Note that if
$u\in{\mathcal{R}}_{1}L_{2}({\mathbb{F}};X)$ and $v\in{\mathcal{R}}_{2}       %
L_{2}({\mathbb{F}};X)$, then both $u$ and $v$ belong to ${\mathcal{R}}       %
L_{2}({\mathbb{F}};X)$, where $r_{\ba}=\min(r_{1,\ba},r_{2,\ba})$. As
usual, the argument $X$ will be omitted if $X={\mathbb{R}}$.

Important particular cases of ${\mathcal{R}}L_{2}({\mathbb{F}};X)$
are
\begin{enumerate}
\item The {\tt sequence spaces} $L_{2,\mfq}({\mathbb{F}};X)$, corresponding to the weights
\[
r_{\ba}= \mfq^{\ba},
\]
where $\mfq=\left\{  q_{k},\, k \geq1\right\}  $ is a sequence of
positive numbers; see \cite{LR-spn,LR_AP,NR}. Given a real number $p$,  one can
also consider the spaces
\begin{equation}
\label{Lpq}
L^p_{2,\mfq}({\mathbb{F}};X)=L_{2,\mfq^p}({\mathbb{F}};X),
\end{equation}
where $\mfq^p=\left\{  q_{k}^p,\, k \geq1\right\}.  $ In particular,
$L^1_{2,\mfq}=L_{2,\mfq}; L^{-1}_{2,\mfq}=L_{2,1/\mfq}$.
Under the additional assumption $q_k\geq 1$
 we have, similar to the usual Sobolev spaces,
$$
L^p_{2,\mfq}({\mathbb{F}};X)\subset L^r_{2,\mfq}({\mathbb{F}};X),\
p>r.
$$
\item The  {\tt Kondratiev  spaces}
 $({\mathcal{S}})_{\rho,\ell}(X)$, corresponding to the weights
\begin{equation}
r_{\ba}=(\ba!)^{\rho/2}(2{\mathbb{N}})^{\ell\ba},\ \rho \in [-1,1],\
 \ell\in \bR,\ \ \label{eq:S-weight}       %
\end{equation}
see \cite{HOUZ}.
\end{enumerate}
There is a natural duality between $L_{2,\mfq}(X)$ and $L^{-1}_{2,\mfq}(X)$:
\begin{equation}
\label{dual-q}
\langle u,v \rangle_{\mfq}=\sum_{\ba\in\cJ}(u_{\ba},v_{\ba})_X;
\end{equation}
there is a natural duality between $(\cS)_{\rho,\ell}(X)$ and $(\cS)_{-\rho,-\ell}(X)$:
\begin{equation}
\label{dual-ks}
\langle u,v \rangle_{\rho,\ell}=\sum_{\ba\in\cJ}(u_{\ba},v_{\ba})_X.
\end{equation}
Both \eqref{dual-q} and \eqref{dual-ks} extend the notion of
$\bE (u,v)_X$ to generalized $X$-valued random elements.

Taking projective and injective limits of weighted spaces leads to constructions
similar to the Schwartz spaces $\mathcal{S}(\bR^d)$ and $\mathcal{S}'(\bR^d)$. Of special interest are
\begin{enumerate}
\item The {\tt power sequence} spaces
\begin{equation}
\label{power-seq-limit}
L_{2,\mfq}^{+}({\mathbb{F}};X)
=\bigcap_{p\in \bR}L^p_{2,\mfq}({\mathbb{F}};X) ,\ \
L_{2,\mfq}^{-}({\mathbb{F}};X)
=\bigcup_{p\in \bR}L^p_{2,\mfq}({\mathbb{F}};X),
\end{equation}
where $\mfq=\{q_k,\,k\geq 1\}$ is a sequence with $q_k\geq 1$
(see \eqref{Lpq}).
\item The spaces $\mathcal{S}^{\rho}(X)$
and $\mathcal{S}_{-\rho}(X)$, $0\leq \rho \leq 1$ of {\tt Kondratiev test functions and distributions}:
\begin{equation}
\label{Kond-limit}
\mathcal{S}^{\rho}(X)=\bigcap_{\ell\in \bR}({\mathcal{S}})_{\rho,\ell}(X),\ \mathcal{S}_{-\rho}(X)=\bigcup_{\ell \in \bR}({\mathcal{S}})_{-\rho,\ell}(X).
\end{equation}
In this regard we mention that, in the traditional white noise setting,
$X=\bR^d$, $\rho=0$ corresponds to the Hida spaces, and the
term {\em Kondratiev spaces} is usually reserved for $\mathcal{S}^{1}(\bR^d)$
and $\mathcal{S}_{-1}(\bR^d)$.
\end{enumerate}
If the space $\cU$ is finite-dimensional, then the sequence $\mfq$ can be taken
finite, with as many elements as the dimension of $\cU$. In this case, certain  Kondratiev spaces are bigger than any sequence space.

\begin{proposition} If $\cU$ is finite-dimensional, then
\begin{equation}
\label{PK-inclus}
L_{2,\mfq}(X)\subset ({\mathcal{S}})_{-\rho,-\ell}(X)
\end{equation}
for every $\rho>0$, $\ell\geq \rho$ and every $\mfq$.
\end{proposition}

\begin{proof}
Let $n$ be the dimension of $\cU$ and
$r=\min\{q_1,\ldots, q_n\}$. Define $\mfr=\{r,\ldots, r\}$.
Then $L_{2,\mfq}(X)\subset L_{2,\mfr}(X)$.
On the other hand, for all $\ba\in \cJ$,
$(2\mathbb{N})^{2}\ba!\geq |\ba|!$,
$$
(r^{2|\ba|}(\ba!)^{\rho}(2\mathbb{N})^{2\rho})^{-1}\leq
(r^{2|\ba|}(|\ba|!)^{\rho})^{-1}\leq C(r)
$$
and therefore $L_{2,\mfr}(X)\subset ({\mathcal{S}})_{-\rho,-\ell}(X)$.
\end{proof}

Analysis of the proof shows that, in general, an inclusion of the type
\eqref{PK-inclus} is possible if and only if there is a uniform in $\ba$ bound
of the type  $(\mfq^{2\ba}(|\ba|!)^{\rho})^{-1}\leq C(2\mathbb{N})^{p\ba}$;
the constants $C$ and $p$ can depend on the sequence $\mfq$. If the space
$\cU$ is infinite-dimensional, then such a bound may exist for certain
sequences $\mfq$ (such as  $\mfq=\mathbb{N}$), and may fail to
exist for other sequences (such as $\mfq=\exp(\mathbb{N}$). Thus,
both $L_{2,\mfq}(X)$ and $({\mathcal{S}})_{-\rho,\ell}(X)$
can be of interest in the study of stochastic differential equations.

\begin{definition}
A {\tt generalized $X$-valued random element} is an element of the set
$\bigcup \cR L_2(\bF;X)$, with the union taken over all weight sequences $\cR$.
\end{definition}

To complete the discussion of weighted spaces, we need the
 following  results about multi-indexed series.
\begin{proposition}
\label{prop:conv-ps}
Let $\mfr=\{r_k,\, k\geq 1\}$ be a sequence of positive numbers.

(1) If $\sum_{k\geq 1} r_k<\infty$, then
\begin{equation}
\label{exp-sum}
\sum_{\ba\in \cJ} \frac{\mfr^{\ba}}{\ba!} = \exp\left(\sum_{k\geq 1} r_k\right).
\end{equation}

(2) If $\sum_{k\geq 1} r_k  < \infty$ and $r_k<1$ for all $k$, then,
for every $\ba\in \cJ$,
\begin{equation}
\label{geom-sum}
\sum_{\bbt\in \cJ} \binom{\ba+\bbt}{\bbt}\mfr^{\bbt}=\left(\prod_{k\geq 1} \frac{1}{1-r_k}\right)\,(1-\mfr)^{-\ba},
\end{equation}
where $1-\mfr$ is the sequence $\{1- r_k,\ k\geq 1\}$.
In particular,
\begin{equation}
\label{geom-sum1}
\sum_{\ba\in \cJ}(2\bN)^{-\ell\ba}<\infty
\end{equation}
for all $\ell>1$; cf. \cite[Proposition 2.3.3]{HOUZ}.

(3) For every $\ba\in \cJ$,
\begin{equation}
\label{binom-sum}
\sum_{\bbt\in \cJ} \binom{\ba}{\bbt}\mfr^{\bbt}=(1+\mfr)^{\ba},
\end{equation}
where $1+\mfr$ is the sequence $\{1+ r_k,\ k\geq 1\}$.
\end{proposition}

\begin{proof}
Note that
\begin{equation*}
\exp\left(\sum_{k\geq 1} r_k\right)=\prod_{k\geq 1} \sum_{n\geq 1}
\frac{r_k^n}{n!},\ \ \
\prod_{k\geq 1} \frac{1}{1-r_k}=\prod_{k\geq 1} \sum_{n\geq 1} r_k^n.
\end{equation*}
By assumption, $\lim_{k\to \infty} r_k=0$, and therefore
$$
\prod_{k\geq 1} r_k^{n_k}=0
$$
unless only finitely many of $n_k$ are not equal to zero. Then both \eqref{exp-sum} and \eqref{geom-sum} with $\ba=\bs{(0)}$ follow.
For general $\ba$,  \eqref{geom-sum} follows from
$$
\sum_{k\geq 0} \binom{n+k}{k}x^k=\frac{1}{(1-x)^{n+1}},\ |x|<1,
$$
which, in turn, follows by differentiating $n$ times the equality $\sum_k x^k=(1-x)^{-1}$. Recall that
$$
\sum_{k}r_k<\infty, \ 0<r_k<1 \ \
\Rightarrow\ \ 0<\prod_k \frac{1}{1-r_k} < \infty.
$$
Equality \eqref{binom-sum} follows from
the usual binomial formula.
\end{proof}

\begin{corollary}
(a) For every collection $f_{\ba},\ \ba\in \cJ$ of elements from $X$ there exists a
weight sequence $r_{\ba},\ \ba \in \cJ$ such that $\sum_{\ba\in{\mathcal{J}}} \Vert f_{\ba
}\Vert_{X}^{2}r_{\ba}^{2}<\infty$.

(b)  If $q_k>1$ and $\sum_{k\geq 1} 1/q_k < \infty$, then the space
$L^+_{2,\mfq}(X)$ is nuclear.

(c) The space $\cS^{\rho}(X)$ is nuclear for every  $\rho\in [0,1]$.
\end{corollary}

\begin{proof} (a) In view of \eqref{geom-sum1}, one can take, for example,
$$
r_{\ba}=\frac{(2\bN)^{-\ba}}{1+\|f_{\ba}\|_X}.
$$

(b) By \eqref{geom-sum}, the embedding
$L^{p+1}_{2,\mfq}(X)\subset L^p_{2,\mfq}(X)$ is Hilbert-Schmidt for every $p\in \bR$.

(c) Note that $\sum_{k\geq 1} (2k)^{-2} < \infty$. Therefore,
by \eqref{geom-sum},
the embedding $(\cS)_{\rho,\ell+1}(X)\subset (\cS)_{\rho,\ell}(X)$ is
Hilbert-Schmidt for every $\ell\in \bR$.
\end{proof}

To summarize, an element $f$ of ${\mathcal{R}}L_{2}({\mathbb{F}};X)$
 can be identified with a formal series
$\sum_{\alpha\in{\mathcal{J}}}f_{\alpha}\xi_{\alpha},$
where $f_{\alpha}\in X$ and
$\sum_{\alpha\in{\mathcal{J}}} \Vert f_{\alpha
}\Vert_{X}^{2}r_{\alpha}^{2}<\infty$. Conversely,
every formal series $\sum_{\ba\in \cJ} f_{\ba} \xi_{\ba}$,
$f_{\ba}\in X$, is a generalized $X$-valued random element.
Using \eqref{Hep-a}, we get an alternative representation of
generalized $X$-valued random elements:
\begin{equation}
\label{eq:gen-Hep}
f=\sum_{\ba\in \cJ} \bar{f}_{\ba} \, \Hep_{\ba},
\end{equation}
with $\bar{f}_{\ba}\in X$.

The following definition extends the three operators of the
Malliavin calculus to generalized random elements.
\begin{definition}
\label{def:main}
Let $u=\sum_{\ba\in \cJ} u_{\alpha}\, \xi_{\ba}$ be a generalized
$\cU$-valued random element,
$v=\sum_{\ba\in \cJ} v_{\alpha}\, \xi_{\ba}$,
 a generalized $X$-valued random element, and
 $f= \sum_{\ba\in \cJ} f_{\ba}\, \xi_{\ba}$,
 a generalized $X\otimes \cU$-valued random element.

(1) The {\tt Malliavin derivative} of $v$ with respect to $u$ is the
generalized $X\otimes \cU$-valued random element
\begin{equation}
\label{MD-gen}
\bD_u(v)=\sum_{\ba\in \cJ}
\left( \sum_{\bbt\in \cJ}\sqrt{\binom{\ba+\bbt}{\bbt}}\
v_{\ba+\bbt}\otimes u_{\bbt}  \right) \xi_{\ba}
\end{equation}
provided the inner sum is well-defined.

(2) The {\tt Skorokhod integral} of $f$ with respect to $u$ is a
generalized $X$-valued random element
\begin{equation}
\label{DO-gen}
\bdel_u(f)=\sum_{\ba\in \cJ}
\left(
\sum_{\bbt\in \cJ}
\sqrt{\binom{\ba}{\bbt}} \ (f_{\bbt},u_{\ba-\bbt})_{\cU}
\right) \xi_{\ba}.
\end{equation}

(3) The {\tt Ornstein-Uhlenbeck} operator with respect to $u$,
when applied to $v$, is a generalized $X$-valued random element
\begin{equation}
\label{OU-gen}
\cL_u(v)=\sum_{\ba\in \cJ}
\left(
\sum_{\bbt,\bg\in \cJ}
\sqrt{\binom{\ba}{\bbt}\binom{\bbt+\bg}{\bbt}} \ v_{\bbt+\bg}
(u_{\bg}, u_{\ba-\bbt})_{\cU}
\right) \xi_{\ba},
\end{equation}
provided the inner sum is well-defined.
\end{definition}
The definitions imply that both $\bD$ and $\bdel$ are {\em bi-linear}
operators:
$$
\cA_{au+bv}(w)=a\cA_u(w)+b\cA_v(w),\
\cA_u(av+bw)=a\cA_u(v)+b\cA_u(w),\ a,b\in \bR,
$$
for all suitable $u,v,w$; $\cA$ is either $\bD$ or $\bdel$.
The operation  $\cL_u(v)$ is linear in $v$ for fixed $u$, but is not
linear in $u$.
The equality
$$
\bD_{\xi_{\bbt}}(\xi_{\ba})=\sqrt{\binom{\ba}{\bbt}}\xi_{\ba-\bbt}
 $$
 shows that, in general $\bD_u(v)\not=\bD_v(u)$.
 The definition of the Skorokhod integral
$\bdel_u(f)$ has a built-in non-symmetry between  the integrator $u$ and
the integrand $f$: they have to belong to different spaces. This is necessary
to keep the definition consistent with \eqref{DO-wn}. Similar non-symmetry
holds for the Ornstein-Uhlenbeck operator $\cL_u(v)$.
Still, we will see later that $\bdel_u(f)=\bdel_f(u)$
if both $f$ and $u$ are real-valued. If $\bD_u(v)$ is defined, then
$\cL_u(v)=\bdel_u(\bD_u(v))$, but $\cL_u(v)$ can exist even when
$\bD_u(v)$ is not defined.

Next, note that
 $\bdel_u(f)$ is a well-defined generalized random element for
 all $u$ and $f$, while definitions of $\bD_u(v)$ and $\cL_u(f)$ require additional
 assumptions. Indeed,
 $\binom{\ba}{\bbt}=0$ unless $\bbt\leq \ba$, and therefore
 the inner sum on the right-hand side of \eqref{DO-gen} always contains
 finitely many non-zero terms. By the same reason, the inner sums on the
 right-hand sides of \eqref{MD-gen}  and \eqref{OU-gen} usually contain
 infinitely many non-zero terms and the convergence must be verified.
 This observation is an extension of
  Remark \ref{rem:main}, and we illustrate it on a concrete example.
 The example also shows that $\cL_u(v)$ can be defined even when $\bD_u(v) $
  is not.
 \begin{example}{\rm
 Consider $u=v=\dot{W}$.
 Then $\bD_u(v)$ is not defined. Indeed,
 \begin{equation}
 \label{ex-coef}
 u_{\ba}=v_{\ba}=
 \begin{cases}
 \mfu_k,&\  {\rm if } \ \ba=\bep(k),\ k\geq 1,\\
 0,& \ {\rm otherwise}.
 \end{cases}
 \end{equation}
 Thus, $\big(\bD_u(v)\big)_{\ba}=0$ if $|\ba|>0$, and
 $$
 \big(\bD_u(v)\big)_{\bs{(0)}}=\sum_{k\geq 1} \mfu_k\otimes \mfu_k,
 $$
 which is not a convergent series.

 On the other hand, interpreting $v$ as an $\bR\otimes \cU$-valued
 generalized random element, we find
 $$
 \big(\bdel_u(v) \big)_{\ba}=
 \begin{cases}
 \sqrt{2}, \ba=2\bep(k),\ k\geq 1,\\
 0,,
 \end{cases}
 $$
 or, keeping in mind that $\sqrt{2}\,\xi_{2\bep(k)}=\Hep_2(\xi_k)$,
 $$
 \bdel_{\dot{W}}(\dotW)=\sum_{k\geq 1} \Hep_2(\xi_k).
 $$
 Note that $\sum_{k\geq 1} \Hep_2(\xi_k)\in (\cS)_{-1,\ell}(\bR)$ for every
 $\ell<-1/2$.

 We conclude the example with an observation that,  although $\bD_{\dotW}(\dotW)$ is not defined, $\cL_{\dotW}(\dot{W})$ is.
 If fact, \eqref{OU-gen} implies that
 $$
 \cL_{\dotW}(\dotW)=\dotW,
 $$
 which is consistent with \eqref{OU-basis} and \eqref{ex-coef}. }
 \end{example}
 If either $u$ or $v$ is a finite linear combination of $\xi_{\ba}$, then
  $\bD_u(v)$ is defined. The following proposition gives two more  sufficient
 conditions for $\bD_u(v)$ to be defined.
 \begin{proposition}
 \label{prop:MD-gc}

 (1) Assume that there exist weights $r_{\ba},\ \ba \in \cJ$ such that
 \begin{equation}
 \label{MDg-v}
 \sum_{\ba\in \cJ} 2^{|\ba|}r_{\ba}^{-2} \|v_{\ba}\|_{X}^2< \infty
 \ {\rm and }
 \ \sum_{\ba\in \cJ}  r_{\ba}^{2} \|u_{\ba}\|_{\cU}^2< \infty.
 \end{equation}
If
 \begin{equation}
 \label{gr-c1}
 \sup_{\bbt\in \cJ}\frac{r_{\ba+\bbt}}{r_{\bbt}}:=b_{\ba}< \infty
 \end{equation}
  for every $\ba\in \cJ$, then $\bD_u(v)$ is well-defined and
  \begin{equation}
  \label{MDa1}
  \|\big(\bD_u(v)\big)_{\ba}\|_{X\otimes \cU}^2 \leq 2^{|\ba|}\,b_{\ba}^2
  \sum_{\bbt\in \cJ} 2^{|\bbt|}r_{\bbt}^{-2} \|v_{\bbt}\|_{X}^2
 \ \sum_{\bbt\in \cJ}  r_{\bbt}^{2} \|u_{\bbt}\|_{\cU}^2\,.
 \end{equation}

 (2) Assume that there exist weights $r_{\ba},\ \ba \in \cJ$ such that
 \begin{equation}
 \label{MDg-u}
\sum_{\ba\in \cJ} r_{\ba}^2 \|v_{\ba}\|_{X}^2< \infty \ {\rm and }
 \ \sum_{\ba\in \cJ} 2^{|\ba|} r_{\ba}^{-2} \|u_{\ba}\|_{\cU}^2< \infty.
 \end{equation}
 If
 \begin{equation}
 \label{gr-c2}
 \sup_{\bbt\in \cJ}\frac{r_{\bbt}}{r_{\ba+\bbt}}:=c_{\ba}< \infty
 \end{equation}
 for every $\ba\in \cJ$, then $\bD_u(v)$ is well-defined and
 \begin{equation}
 \label{MDa2}
 \|\big(\bD_u(v)\big)_{\ba}\|_{X\otimes \cU}^2\leq 2^{|\ba|}\,c_{\ba}^2\,
 \sum_{\bbt\in \cJ} r_{\bbt}^2 \|v_{\bbt}\|_{X}^2
  \sum_{\bbt\in \cJ} 2^{|\bbt|} r_{\bbt}^{-2} \|u_{\bbt}\|_{\cU}^2\,.
 \end{equation}
 \end{proposition}

 \begin{proof}
 Using
 $$
 \sum_{k\geq 0} \binom{n}{k}=2^n
 $$
 we conclude that $\binom{n}{k}\leq 2^n$ for all $k\geq 0$ and therefore
 \begin{equation}
 \label{multi-binom}
 \binom{\ba}{\bbt} =\prod_k \binom{\alpha_k}{\beta_k}\leq  2^{|\ba|}
 \end{equation}
 for all $\bbt\in \cJ$.
 Therefore,
 \begin{equation}
 \begin{split}
 \big\|\big(\bD_u(v)\big)_{\ba}\big\|_{X\otimes \cU} &= \left\| \sum_{\bbt\in \cJ}\sqrt{\binom{\ba+\bbt}{\bbt}}\
v_{\ba+\bbt}\otimes u_{\bbt}  \right\|_{X\otimes \cU}\\
&\leq \sum_{\bbt\in \cJ} 2^{|\ba+\bbt|/2}\,\|v_{\ba+\bbt}\|_X\ \|u_{\bbt}\|_{\cU}.
\end{split}
\end{equation}
 and the result follows by the Cauchy-Schwartz inequality.
 \end{proof}

\begin{remark} {\rm (a) If $r_{\ba}=\mfq^{\ba}$ for some sequence
$\mfq$, then both \eqref{gr-c1} and \eqref{gr-c2} hold. (b) More information
about the structure of $u$ and/or $v$ can lead to weaker sufficient conditions.
 For example, if
 $(u_{\ba},u_{\bbt})_{\cU}=0$ for $\ba\not=\bbt$, and
 $\|u_{\ba}\|_{\cU}\leq 1$, then
 $\big\|\big(\bD_u(v)\big)_{\ba}\big\|_{X\otimes \cU}^2<\infty$ if
 and only if
 $$
 \sum_{\bbt\in \cJ}\binom{\ba+\bbt}{\bbt}\|v_{\ba+\bbt}\|_X^2<\infty,
 $$
 which is a generalization of \eqref{MD-cond}. Similarly, if $(u_{\ba},u_{\bbt})_{\cU}=0$ for $\ba\not=\bbt$,  then $\big(\cL_u(v)\big)_{\ba}$ exists for all $\ba\in \cJ$ and
 $$
 \big(\cL_u(v)\big)_{\ba}=\sum_{\beta\in \cJ}\binom{\ba}{\bbt}\, v_{\bbt}
 \, \|u_{\ba-\bbt}\|_{\cU}^2.
 $$}
\end{remark}

The reader is encouraged to verify that
\begin{enumerate}
\item If  $u=\dot{W}$, with $u_{\bep(k)}=\mfu_k$ and $u_{\ba}=0$ otherwise, then
 \eqref{MD-gen}, \eqref{DO-gen}, and \eqref{OU-gen} become, respectively,
 \eqref{MD-wn}, \eqref{DO-wn}, and \eqref{OU-wn}.
 \item The operators $\boldsymbol{\delta}_{\xi_{k}}$ and $\mathbf{D}_{\xi_{k}}$ are the \emph{creation} and \emph{annihilation} operators from quantum physics \cite{GJ}:
\begin{equation}
{\mathbf{D}}_{\xi_k}(\xi_{\boldsymbol{\alpha}})=\sqrt{\alpha_{k}}       %
\,\xi_{\boldsymbol{\alpha}-\boldsymbol{\varepsilon}\left(  k\right)  },\ \ {\boldsymbol{\delta}}_{\xi_{k}}(\xi_{\boldsymbol{\alpha}       %
})=\,\sqrt{\alpha_{k}+1}\,\xi_{\boldsymbol{\alpha}+\boldsymbol{\varepsilon
}\left(  k\right).  }
\label{eq: uxi}
\end{equation}
More  generally,
\begin{equation}
\label{MC-basis}
\bD_{\xi_{\bbt}}(\xi_{\ba})=\sqrt{\binom{\ba}{\bbt}}\,\xi_{\ba-\bbt},\
\bdel_{\xi_{\bbt}}(\xi_{\ba})=\sqrt{\binom{\ba+\bbt}{\bbt}}\, \xi_{\ba+\bbt},\ \cL_{\xi_{\bbt}}(\xi_{\ba})={\binom{\ba}{\bbt}}\,\xi_{\ba}.
\end{equation}
\item If
\begin{equation}
\label{dual-squit-cond}
v\in L_2(\bF;X), \ f\in L_2(\bF;X\otimes \cU),\
\bD_u(v)\in L_2(\bF;X\otimes \cU),\
\bdel_u(f)\in L_2(\bF;X),
\end{equation}
then a simple rearrangement of terms shows that the
      following analogue of \eqref{eq:SkI-ip} holds:
\begin{equation}
\label{dual-squit}
\bE\big(\bD_u(v),f\big)_{X\otimes \cU}=\bE \big(v,\bdel_u(f)\big)_X.
      \end{equation}
       For example,
$$
\bD_u(\xi_{\bs{\gamma}})=
\sum_{\ba\in \cJ}
\sqrt{\binom{\bg}{\ba}}\
u_{\bs{\gamma}-\ba}\, \xi_{\ba},
$$
and, if we assume that $u$ and $f$ are such that $\bdel_u(f)\in L_2(\bF;X)$,
then
$$
\bE \big(\xi_{\bs{\gamma}}\bdel_u(f)\big)=
\sum_{\ba\in \cJ}
\sqrt{\binom{\bg}{\ba}}\
(u_{\bs{\gamma}-\ba},f_{\ba})_{\cU}=
\bE\big(f,\bD_u(\xi_{\bs{\gamma}}\big)_{\cU}.
$$
\item With the notation $\Hep_{\ba}=\sqrt{\ba!}\,\xi_{\ba}$,
\[
{\mathbf{D}}_{\xi_{k}}(\Hep_{\boldsymbol{\alpha}})=\alpha_{k}       %
\,\Hep_{\boldsymbol{\alpha}-\boldsymbol{\varepsilon}\left(  k\right)},\ \ {\boldsymbol{\delta}}_{\xi_{k}}(\Hep_{\boldsymbol{\alpha}})
=\Hep_{\boldsymbol{\alpha}+\boldsymbol{\varepsilon}\left(  k\right)  },
\]
and
\begin{equation}
\label{wp1}
\bdel_{\Hep_{\ba}}(\Hep_{\bbt})=\Hep_{\ba+\bbt}.
\end{equation}
\end{enumerate}

To conclude the section, we use \eqref{wp1} to establish a connection between the
Skorokhod integral $\bdel$ and the {\em Wick product.}

\begin{definition}
\label{def:wp}
Let $f$ be a generalized $X$-valued random element and $\eta$,
a generalized real-valued random element. The
{\tt Wick product} $f\diamond \eta$ is a generalized $X$-valued random
element defined by
\begin{equation}
\label{wp-def}
f\diamond \eta = \sum_{\ba\in \cJ} \left(
\sum_{\bbt\in \cJ} \sqrt{\binom{\ba}{\bbt}}
f_{\ba-\bbt}\,\eta_{\bbt}\right)\ \xi_{\ba}.
\end{equation}
\end{definition}
The definition implies that $f\diamond \eta=\eta\diamond f$,
\begin{equation}
\label{wp-bas}
\xi_{\ba}\diamond \xi_{\bbt}=\sqrt{\binom{\ba+\bbt}{\ba}}\,\xi_{\ba+\bbt},\ \ \Hep_{\ba}\diamond \Hep_{\bbt}=\Hep_{\ba+\bbt}.
\end{equation}
In other words,
 \eqref{wp-def}  extends by linearity relation \eqref{wp1} to generalized random elements.
Comparing  \eqref{wp-def} and \eqref{DO-gen}, we get the
connection between the Wick product and the Skorokhod integral.
\begin{theorem}
\label{th:wp-do}
If $f$ is a generalized $X$-valued random element and $\eta$,
a generalized real-valued random element, then $\bdel_{\eta}(f)=
f\diamond \eta$.
In particular, if
$\eta$ and $\theta$ are generalized real-valued random elements, then
$$
\bdel_{\eta}(\theta)=\bdel_{\theta}(\eta)=\eta\diamond \theta.
$$
\end{theorem}

The original definition of Wick product \cite{Wick-orig} is not related to the Skorokhod integral, and is it remarkable that the two coincide in some situations. The important feature of \eqref{wp-def} is the presence of point-wise
multiplication, which does not admit a straightforward extension to
general spaces.

 A natural definition of the multiple Wiener-It\^{o}
integral in the one-dimensional setting, that is, with respect to a single
standard Gaussian random variable $\xi$, is as follows.
With only scalar as possible integrands, set
$$
I_n(1) = \Hep_n(\xi).
$$
As expected,
$$
\bD_{\xi} (I_n(1))=n I_{n-1}(1),\ \bdel_{\xi}(I_n(1))=I_{n+1}(1).
$$
This  is consistent with the general definition as long as the Wick product is used throughout: $I_n(1) =\xi^{\diamond n}$.
An interested reader can easily extend this construction to finitely many
iid  standard Gaussian random variables.

Definition \ref{def:wp} and Theorem \ref{th:wp-do} raise the following
questions:
\begin{enumerate}
\item Is it possible to extend the operation $\diamond$ by replacing the
point-wise product on the right-hand side of \eqref{wp-def} with
 something else and still preserve the connection with the operator $\bdel$? Clearly, simply setting $f\diamond u=\bdel_u(f)$ is not
 acceptable, as we expect the $\diamond$ operation to be fully
 symmetric.
\item Under what conditions will the operator $v\mapsto u\diamond v$
 be  (a Hilbert space) adjoint or (a topological space) dual of $\bD_u$?
 \item What is the most general construction of the multiple  Wiener-It\^{o}
 integral?
 \end{enumerate}
 We will not address these questions in this paper and leave them for future
  investigation (see references
 \cite{PotAn1, PotAn2} for some particular cases).


\section{Elements of Malliavin Calculus on special spaces}
\label{sec:ex}

The  objectives of this section are
\begin{itemize}
\item  to establish results of the type
$$
\|\cA_u(v)\|_a\leq C\big(\|u\|_b\big) \,\|v\|_c,
$$
where $\|\cdot\|_i,\ i=a,b,c$ are norms in the  suitable sequence
or Kondratiev spaces, the function $C$ is  independent of  $v$, and
$\cA$ is one of the operators $\bD$, $\bdel$, $\cL$.
\item to look closer at  $\bD$ and  $\bdel$
as adjoints of each other when \eqref{dual-squit-cond} does not hold.
\end{itemize}
To simplify the notations, we will write
$L^p_{2,\mfq}(X)$ for $L^p_{2,\mfq}(\bF;X)$.

We start with the ``path of the least resistance'' approach and see what  one
can obtain with a straightforward application of the Cauchy-Schwartz inequality.
The first collection of results is for the sequence spaces.
\begin{theorem}
\label{th:main-ps}
Let $\mfq=\{q_k,\ k\geq 1\}$ be a sequence such that $q_k>1$ for all $k$ and
$\sum_{k\geq 1} 1/q_k^2 < \infty$. Denote by $\sqrt{2}\mfq$ the
sequence $\{\sqrt{2}\,q_k,\ k\geq 1\}$.

(a) If $u\in L^{-1}_{2,\mfq}(\cU)$ and $v\in L_{2,\sqrt{2}\mfq}(X)$, then
$\bD_u(v)\in L_{2}(\bF;X\otimes \cU)$ and
$$
\big(\bE\|\bD_u(v)\|_{X\otimes \cU}^2\big)^{1/2} \leq \left(\prod_{k\geq 1} \frac{q_k^2}{q_k^2-1}\right)^{1/2}\,
\|u\|_{L^{-1}_{2,\mfq}(\cU)}\, \|v\|_{L_{2,\sqrt{2}\mfq}(X)}.
$$
(b) If $u\in L^{-1}_{2,\mfq}(\cU)$, $f\in L^{-1}_{2,\mfq}(X\otimes \cU)$,
and $\sum_{k\geq 1} 2^k/q_k^2<\infty$, then $\bdel_u(f)\in L^{-1}_{2,\sqrt{2}\mfq}(X)$ and
$$
\|\bdel_u(f)\|_{L^{-1}_{2,\sqrt{2}\mfq}(X)}\leq \left(\sum_{k\geq 1}\frac{2^k}{q_k^2} \right)^{1/2}
\|u\|_{ L^{-1}_{2,\mfq}(\cU)}\, \|f\|_{L^{-1}_{2,\mfq}(X\otimes \cU)}.
$$
In particular, if $u\in L^-_{2,\mfq}(\cU)$ and $f\in L^-_{2,\mfq}(X\otimes U)$, then $\bdel_u(f)\in L^-_{2,\mfq}(X)$.

(c) If $u\in L^{-1}_{2,\mfq}(\cU)$, $v\in L_{2,\sqrt{2}\mfq}(X)$, and $\sum_{k\geq 1} 2^k/q_k^2<\infty$, then $\cL_u(v)\in L^{-1}_{2,\sqrt{2}\mfq}(X)$ and
$$
\|\cL_u(v)\|_{L^{-1}_{2,\sqrt{2}\mfq}(X)}\leq \left(\prod_{k\geq 1} \frac{q_k^2}{q_k^2-1}\right)^{1/2}\,
\left(\sum_{k\geq 1} \frac{2^k}{q_k^2} \right)^{1/2}
\|u\|^2_{L^{-1}_{2,\mfq}(\cU)}\, \|v\|_{L_{2,\sqrt{2}\mfq}(X)}.
$$
\end{theorem}

\begin{proof}
(a) By \eqref{MDa1} with
$r_{\ba}=b_{\ba}=\mfq^{-\ba}$,
$$
\|\big(\bD_u(v)\big)_{\ba}\|_{X\otimes \cU}^2
\leq \mfq^{-2\ba}\,\|u\|_{L^{-1}_{2,\mfq}(\cU)}^2\, \|v\|_{L_{2,\sqrt{2}\mfq}(X)}^2
$$
The result then follows from \eqref{geom-sum}.

(b) By   \eqref{DO-gen},  \eqref{multi-binom},
 and the Cauchy-Schwartz inequality,
$$
\|\big(\bdel_u(f)\big)_{\ba}\|_{X}^2 \leq
2^{|\ba|}\mfq^{2\ba}\sum_{\bbt}\mfq^{-2\bbt}\|f_{\bbt}\|_{X\otimes \cU}^2
\sum_{\bbt\leq \ba}\mfq^{-2(\ba-\bbt)}\|u_{\ba-\bbt}\|_{\cU}^2,
$$
and the result follows.

(c) This follows by combining the results of (a) and (b).
\end{proof}
Analysis of the proof shows that alternative results are possible by avoiding inequality \eqref{multi-binom}; see Theorem \ref{th:main-ps1} below.
The next collection of results, this time for the Kondratiev spaces,  is again in the spirit of the ``path of the least resistance.''
\begin{theorem}
\label{th:main-ks}

(a) If $u\in (\cS)_{-1,-\ell}(\cU)$ and $v\in (\cS)_{1,\ell}(X)$ for some
 $\ell\in \bR$, then $\bD_u(v)\in (\cS)_{1,\ell-p}(X\otimes \cU)$ for
 all $p>1/2$, and
 $$
 \|\bD_u(v)\|_{(\cS)_{1,\ell-p}(X\otimes \cU)}^{1/2}\leq
 \left(\prod_{k\geq 1} \frac{1}{1-(2k)^{-2p}}\right)^{1/2}
 \|u\|_{ (\cS)_{-1,-\ell}(\cU)}\, \|v\|_{ (\cS)_{1,\ell}(X)}.
 $$
(b)  If $u\in (\cS)_{-1,\ell}(\cU)$ and $f\in (\cS)_{-1,\ell}(X\otimes \cU)$ for some
 $\ell \in \bR$, then $\bdel_u(f)\in (\cS)_{-1,\ell-p}(X)$ for every $p>1/2$, and
 $$
 \|\bdel_u(f)\|_{(\cS)_{-1,\ell-p}(X)}\leq
 \left(\sum_{\ba\in \cJ} (2\mathbb{N})^{-2p\ba}\right)^{1/2}
 \|u\|_{ (\cS)_{-1,\ell}(\cU)}\, \|f\|_{ (\cS)_{-1,\ell}(X\otimes \cU)}.
 $$
 In particular, if $u\in \cS_{-1}(\cU)$ and $f\in \cS_{-1}(X\otimes \cU)$, then
 $\bdel_u(f)\in \cS_{-1}(X)$.

(c) If $u\in (\cS)_{-1,-\ell}(\cU)$ and $v\in (\cS)_{1,\ell+p}(X)$ for some
 $\ell\in \bR$ and $p>1/2$, then $\cL_u(f)\in (\cS)_{-1,\ell-p}(X)$
 \begin{equation*}
 \begin{split}
 \|\cL_u(v)\|_{(\cS)_{-1,\ell-p}(X)}&\leq
 \left(\prod_{k\geq 1} \frac{1}{1-(2k)^{-2p}}\right)^{1/2}
 \left(\sum_{\ba\in \cJ} (2\mathbb{N})^{-2p\ba}\right)^{1/2}\\
 &
 \|u\|^2_{ (\cS)_{-1,-\ell}(\cU)}\, \|v\|_{ (\cS)_{1,\ell}(X\otimes \cU)}.
 \end{split}
 \end{equation*}
\end{theorem}

\begin{proof}
To simplify the notations, we write  $r_{\ba}=(2\mathbb{N})^{\ell \ba}$.

(a) By \eqref{MD-gen},
$$
\big(\bD_u(v)\big)_{\ba}=\sum_{\bbt}
\left(\frac{r^2_{\ba+\bbt}(\ba+\bbt)!}{r^2_{\ba}\, r^2_{\bbt}
\ba!\bbt!}\right)^{1/2}v_{\ba+\bbt}\otimes u_{\bbt}.
$$
To get the result, use triangle inequality, followed by the Cauchy-Schwartz inequality and \eqref{exp-sum}.

(b) By   \eqref{DO-gen}
 and the Cauchy-Schwartz inequality,
$$
\|\big(\bdel_u(f)\big)_{\ba}\|_{X}^2 \leq
r^{-2}_{\ba}\ba!\sum_{\bbt}\frac{r_{\bbt}^2}{\bbt!}
\|f_{\bbt}\|_{X\otimes \cU}^2
\sum_{\bbt\leq \ba}\frac{r_{\ba-\bbt}^2}{(\ba-\bbt)!}\|u_{\ba-\bbt}\|_{\cU}^2,
$$
and the result follows.

(c) This follows by combining the results of (a)  and (b), because
$$
(\cS)_{1,\ell}(X)\subset (\cS)_{-1,\ell}(X).
$$
\end{proof}

Let us now discuss the duality relation between $\bdel_u$ and $\cD_u$.
Recall that \eqref{dual-squit} is just a consequence of the definitions, once the
terms in the corresponding sums are rearranged, {\em as long as the sums converge}. Condition \eqref{dual-squit-cond} is one way to ensure the convergence, but is not the only possibility: one can also use duality relations
between various weighted chaos spaces.

In particular,  duality relation \eqref{dual-ks} and  Theorem \ref{th:main-ks}
lead to the following version of \eqref{dual-squit}: if, for
some $\ell\in \bR$ and $p>1/2$, we have  $u\in (\cS)_{-1,-\ell-p}(\cU)$,
$v\in (\cS)_{1,\ell+p}(X)$, and $f\in (\cS)_{-1,\ell}(X\otimes \cU)$,  then
\begin{equation}
\label{dualOP-KS}
\langle \bdel_u(f),v\rangle_{1,\ell+p}=\langle f,\bD_u(v)\rangle_{1,\ell}.
\end{equation}
To derive a similar result in the sequence spaces,
we need a different version of Theorem \ref{th:main-ps}.
\begin{theorem}
\label{th:main-ps1}
Let $\mfp,\ \mfq,$ and $\mfr$ be sequences of positive numbers such that
\begin{equation}
\label{ps-rel}
\frac{1}{p_k^2}+\frac{1}{q_k^2}=\frac{1}{r_k^2},\ k\geq 1.
\end{equation}
(a) If $u\in L_{2,\mfp}(\cU)$ and $f\in L_{2,\mfq}(X\otimes \cU)$, then
$\bdel_u(f)\in L_{2,\mfr}(X)$ and
$$
\|\bdel_u(f)\|_{L_{2,\mfr}(X)}\leq \|u\|_{L_{2,\mfp}(\cU)}\,
\|f\|_{L_{2,\mfq}(X\otimes \cU)}.
$$
(b) In addition to \eqref{ps-rel} assume that
\begin{equation}
\label{ps-rel1}
\sum_{k\geq 1} \frac{r_k^2}{p_k^2}<\infty.
\end{equation}
Define
$$
\bar{C}=\left(\prod_{k\geq 1} \frac{p_k^2}{p_k^2-r_k^2}\right)^{1/2}.
$$
If $u\in L_{2,\mfp}(\cU)$ and $v\in L^{-1}_{2,\mfr}(X)$, then
 $\bD_u(v)\in L^{-1}_{2,\mfq}(X\otimes \cU)$ and
 $$
 \|\bD_u(v)\|_{L^{-1}_{2,\mfq}(X\otimes \cU)} \leq
 \bar{C}\,\|u\|_{ L_{2,\mfp}(\cU)}\, \|v\|_{ L^{-1}_{2,\mfr}(X)}.
 $$
\end{theorem}
\begin{proof}
(a) By \eqref{DO-gen},
\begin{equation*}
\begin{split}
\|\bdel_u(f)\|_{L_{2,\mfr}(X)}^2&=\sum_{\bg\in\cJ}
\left\|\sum_{\ba+\bbt=\bg}\sqrt{\binom{\bg}{\ba}}
(f_{\ba},u_{\bbt})_{\cU}\right\|^2_X \mfr^{2\bg} \\
&\leq
\sum_{\bg\in\cJ}\left\|\sum_{\ba+\bbt=\bg}
\sqrt{\binom{\bg}{\ba}}|(f_{\ba},u_{\bbt})_{\cU}|\;
\mfr^{\ba}\mfr^{\bbt}\right\|^2_X.
\end{split}
\end{equation*}
Define the sequence $\mfc=\{c_k,\ k\geq 1\}$ by $c_k=p^2_k/q^2_k$,
so that
\begin{equation}
\label{aux-d1}
(1+\mfc^{-1})^{\ba}\mfr^{2\ba}=\mfq^{2\ba},\
(1+\mfc)^{\ba}\mfr^{2\ba}=\mfp^{2\ba}.
\end{equation}
Then
$$
\|\bdel_u(f)\|_{L_{2,\mfr}(X)}^2\leq \sum_{\bg\in\cJ}\left(\sum_{\ba+\bbt=\bg}
\sqrt{\binom{\bg}{\ba}}\mfc^{\ba/2}\mfc^{-\ba/2}
\|f_{\ba}\|_{\cU\otimes X}\|u_{\bbt}\|_{\cU}\;
\mfr^{\ba}\mfr^{\bbt}\right)^2.
$$
By the Cauchy-Schwartz inequality and \eqref{binom-sum},
\begin{equation*}
\begin{split}
\|\bdel_u(f)\|_{L_{2,\mfr}(X)}^2\leq &
\sum_{\bg\in\cJ}\left(\left(\sum_{\ba\in \cJ}{\binom{\bg}{\ba}}\mfc^{\ba}\right)
\left(\sum_{\ba+\bbt=\bg}\mfc^{-\ba}
\|f_{\ba}\|^2_{\cU\otimes X}\|u_{\bbt}\|^2_{\cU}
\;\mfr^{2\ba}\mfr^{2\bbt}\right)\right)\\
= &
\sum_{\bg\in\cJ}\left((1+\mfc)^{\bg}
\left(\sum_{\ba+\bbt=\bg}\mfc^{-\ba}
\|f_{\ba}\|^2_{\cU\otimes X}\|u_{\bbt}\|^2_{\cU}\;
\mfr^{2\ba}\mfr^{2\bbt}\right)\right) \\
= &
\sum_{\bg\in\cJ}\sum_{\ba+\bbt=\bg}(1+\mfc^{-1})^{\ba}
(1+\mfc)^{\bbt}
\|f_{\ba}\|^2_{\cU\otimes X}\|u_{\bbt}\|^2_{\cU}\;
\mfr^{2\ba}\mfr^{2\bbt}\\
= & \left(\sum_{\ba\in\cJ}
\|f_{\ba}\|^2_{\cU\otimes X}\;(1+\mfc^{-1})^{\ba}\mfr^{2\ba}\right)
\left(\sum_{\bbt\in\cJ}\|u_{\bbt}\|^2_{\cU}\;(1+\mfc)^{\bbt}\mfr^{2\bbt}
\right)\\
=& \left(\sum_{\ba\in\cJ}\|f_{\ba}\|^2_{\cU\otimes X}\;\mfp^{2\ba}\right)
\left(\sum_{\bbt\in\cJ}\|u_{\bbt}\|^2_{\cU}\;\mfq^{2\bbt}\right)=
\|f\|^2_{L_{2,\mfp}(X\otimes \cU)}\;\|u\|^2_{L_{2,\mfp}(\cU)}.
\end{split}
\end{equation*}
(b) By \eqref{MD-gen},
\begin{equation*}
\begin{split}
 \|\bD_u(v)\|_{L^{-1}_{2,\mfq}(X\otimes \cU)}^2&=
 \sum_{\ba\in \cJ} \left\|\sum_{\bbt\in \cJ}
 \sqrt{\binom{\ba+\bbt}{\bbt} } \
 v_{\ba+\bbt}\otimes u_{\bbt}\right\|_{X\otimes \cU}^2
 \mfq^{-2\ba}\\
 &\leq \sum_{\ba\in \cJ} \left(\sum_{\bbt\in \cJ}
 \sqrt{\binom{\ba+\bbt}{\bbt} } \
 \|v_{\ba+\bbt}\|_{X}\, \| u_{\bbt}\|_{\cU} \right)^2
 \mfq^{-2\ba}
\end{split}
\end{equation*}
 Define the sequence $\mfc=\{c_k,\ k\geq 1\}$ by $c_k=r^2_k/p^2_k<1$,
so that
\begin{equation}
\label{aux-d1a}
(\mfc^{-1}-1)^{\ba}\mfq^{2\ba}=\mfp^{2\ba},\
(1-\mfc)^{\ba}\mfq^{2\ba}=\mfr^{2\ba}.
\end{equation}
 Then
 $$
 \|\bD_u(v)\|_{L^{-1}_{2,\mfq}(X\otimes \cU)}^2\leq
 \sum_{\ba\in \cJ} \left(\sum_{\bbt\in \cJ}
 \sqrt{\binom{\ba+\bbt}{\bbt} } \
 \mfc^{\bbt/2}\mfq^{-(\ba+\bbt)}\mfc^{-\bbt/2}\|v_{\ba+\bbt}\|_{X}\, \mfq^{\bbt} \|u_{\bbt}\|_{\cU} \right)^2.
 $$
 By the Cauchy-Schwartz inequality and \eqref{geom-sum},
 \begin{equation*}
\begin{split}
\|\bD_u(v)\|_{L^{-1}_{2,\mfq}(X\otimes \cU)}^2&\leq
\sum_{\ba\in \cJ}\left(\sum_{\bbt\in \cJ} \binom{\ba+\bbt}{\bbt} \mfc^{\bbt}
\right) \left(\sum_{\bbt\in \cJ}
\mfq^{-2(\ba+\bbt)}\mfc^{-\bbt}\|v_{\ba+\bbt}\|_{X}^2\, \mfq^{2\bbt} \|u_{\bbt}\|_{\cU}^2
\right)\\
&=
\bar{C}^2\sum_{\ba\in \cJ}\left(\Big( (1-\mfc)^{-\ba}\Big) \sum_{\bbt\in \cJ}
\|v_{\ba+\bbt}\|_X^2\mfc^{-\bbt}\, \mfq^{-2(\ba+\bbt)}\,
\mfq^{2\bbt}\, \|u_{\bbt}\|_{\cU}^2\right)\\
=
\bar{C}^2\sum_{\bbt\in \cJ}\Bigg(\|u_{\bbt}\|_{\cU}^2 &
 \,\mfc^{-\bbt} (1-\mfc)^{\bbt}\,\mfq^{2\bbt}
\left(\sum_{\ba\in \cJ} \|v_{\ba+\bbt}\|_X^2
\mfq^{-2(\ba+\bbt)} (1-\mfc)^{-(\ba+\bbt)}\right)\Bigg)\\
&\leq
\bar{C}^2\left(\sum_{\bbt\in \cJ}\|u_{\bbt}\|_{\cU}^2
(\mfc^{-1}-1)^{\bbt}\mfq^{2\bbt}\right)
\left(\sum_{\ba\in \cJ}\|v_{\ba}\|_X^2
(1-\mfc)^{-\ba}\mfq^{-2\ba}\right)\\
&=
\bar{C}^2\|u\|^2_{ L_{2,\mfp}(\cU)}\, \|v\|^2_{ L^{-1}_{2,\mfr}(X)},
\end{split}
\end{equation*}
where the last equality follows from \eqref{aux-d1a}. Note also that
$$
\sum_{\ba\in \cJ} \|v_{\ba+\bbt}\|_X^2
\mfr^{-2(\ba+\bbt)}
\leq
\|v\|^2_{ L^{-1}_{2,\mfr}(X)}
$$
and the equality holds if and only if $\bbt=\bs{(0)}$.
\end{proof}

Together with duality relation \eqref{dual-q},  Theorem \ref{th:main-ps1}
leads to the following version of \eqref{dual-squit-cond}:  if $u\in L_{2,\mfp}(\cU)$,
 $f\in L_{2,\mfq}(X\otimes \cU)$,  and $v\in L^{-1}_{2,\mfr}(X)$, if the
sequences $\mfp, \mfq, \mfr$ are related by \eqref{ps-rel}, and
if  \eqref{ps-rel1} holds, then
\begin{equation}
\label{dualOP-PS}
\langle \bdel_u(f),v\rangle_{\mfr}=\langle f,\bD_u(v)\rangle_{\mfq}.
\end{equation}
Here is a general procedure to construct sequences $\mfp, \mfq, \mfr$
satisfying \eqref{ps-rel} and \eqref{ps-rel1}.
Start with an arbitrary sequence of positive numbers $\mfp$ and a sequence
$\mfc$ such that $0<c_k<1$ and $\sum_{k\geq 1} c_k<\infty$.
Then set $r_k^2=c_kp_k^2$ and $q_k^2=p_k^2/(c_k^{-1}-1)$.
If the space $\mathcal{U}$ is $n$-dimensional, then condition
\eqref{ps-rel1} is not necessary because
  sequences $\mfp,\mfq,\mfr$  are finite.

  \begin{theorem}
  \label{OU-PSp}
  Let $\mfp,\ \mfq,$ and $\mfr$ be sequences of positive numbers such that
\begin{align}
\label{ps-ou1}
&\left(\frac{1}{r_k^2}-\frac{1}{p_k^2}\right)\left(q_k^2-\frac{1}{p_k^2}\right)=1,\ k\geq 1;\\
\label{ps-ou2}
&p_k^2q_k^2>1,\ k\geq 1, \ {\rm and\ } \sum_{k\geq 1}\frac{1}{p_k^2q_k^2} < \infty.
\end{align}
 If $u\in L_{2,\mfp}(\cU)$ and $v\in L_{2,\mfq}(X)$, then
$\cL_u(v)\in L_{2,\mfr}(X)$ and
\begin{equation}
\label{OU-cnorm}
\|\cL_u(v)\|_{L_{2,\mfr}(X)}\leq
\left(\prod_{k\geq 1} \frac{p_k^2q_k^2}{p_k^2q_k^2-1}\right)^{1/2}
 \|u\|_{L_{2,\mfp}(\cU)}^2\,
\|v\|_{L_{2,\mfq}(X)}.
\end{equation}
  \end{theorem}

  \begin{proof}
  It follows from \eqref{OU-gen} that
  $$
  \|\big(\cL_u(v)\big)_{\ba}\|_X^2 \leq
  \left(
  \sum_{\bbt,\bg} \sqrt{\binom{\bbt+\bg}{\bg}\, \binom{\ba}{\bbt}}
  \|v_{\bbt+\bg}\|_X\ \|u_{\ba-\bbt}\|_{\cU}\ \|u_{\bg}\|_{\cU} \right)^{2}.
  $$
  Let $\mathfrak{h}=\{h_k,\ k\geq 1\}$ be a sequence of positive numbers
  such that  that
  $$
 h_k<1,\ k\geq 1,\  \sum_k h_k<\infty.
 $$
 Define
 $$
 C_h=\left(\prod_{k}\frac{1}{1-h_k}\right)^{1/2}.
 $$
 Then
 \begin{equation*}
 \begin{split}
 \sum_{\bg} \sqrt{\binom{\bbt+\bg}{\bg}}
  \|v_{\bbt+\bg}\|_X\|u_{\bg}\|_{\cU}& \leq
  \left(\sum_{\bg} \binom{\bbt+\bg}{\bg} \mathfrak{h}^{\bg}\right)^{1/2}
  \left(\sum_{\bg}\mathfrak{h}^{-\bg}\|v_{\bbt+\bg}\|^2_X\|u_{\bg}\|^2_{\cU}
  \right)^{1/2}\\
  &=C_h\left(\frac{1}{(1-\mathfrak{h})^{\bbt}}\right)^{1/2}
  \left(\sum_{\bg}\mathfrak{h}^{-\bg}\|v_{\bbt+\bg}\|^2_X\|u_{\bg}\|^2_{\cU}
  \right)^{1/2}
 \end{split}
 \end{equation*}
 Next, take another sequence $\mathfrak{w}=\{w_k,\ k\geq 1\}$ of positive numbers and define the sequence $\mfc=\{c_k,\ k\geq 1\} $ by
 \begin{equation}
 \label{OU-c}
 c_k=\frac{w_k}{1-h_k}.
 \end{equation}
Then
\begin{equation*}
 \begin{split}
& \left(
  \sum_{\bbt,\bg} \sqrt{\binom{\bbt+\bg}{\bg}\, \binom{\ba}{\bbt}}
  \|v_{\bbt+\bg}\|_X\|u_{\ba-\bbt}\|_{\cU}\|u_{\bg}\|_{\cU} \right)^{2}
 \\ & \qquad \leq
  C_h^2\left(\sum_{\bbt} \binom{\ba}{\bbt}\mfc^{\bbt}\right)
 \left(  \sum_{\bbt\leq \ba}\|u_{\ba-\bbt}\|_{\cU}^2\mathfrak{w}^{-\bbt}
\left(\sum_{\bg}\mathfrak{h}^{-\bg}\|v_{\bbt+\bg}\|^2_X\|u_{\bg}\|^2_{\cU}
  \right)\right).
\end{split}
 \end{equation*}
 As a result,
\begin{equation*}
 \begin{split}
 \sum_{\ba} \|\big(\cL_u(v)\big)_{\ba}\|_X^2 \mfr^{2\ba}
  \leq
  &C_h^2\sum_{\bg} \mfr^{-2\bg}(1+\mfc)^{-\bg} \mathfrak{w}^{\bg}
  \mathfrak{h}^{-\bg}\|u_{\bg}\|^2_{\cU}\\
 &\sum_{\bbt} \mfr^{2(\bbt+\bg)}(1+\mfc)^{\bbt+\bg} \mathfrak{w}^{-(\bbt+\bg)}\|v_{\bbt+\bg}\|_X^2\\
& \sum_{\ba} \mfr^{2(\ba-\bbt)}(1+\mfc)^{\ba-\bbt} \|u_{\ba-\bbt}\|_{\cU}^2
  \end{split}
 \end{equation*}
 Then \eqref{OU-cnorm} holds if
 \begin{equation}
 \label{ou-eq1-3}
 r_k^2(1+c_k)=\frac{w_k}{r_k(1+c_k)h_k}=p_k^2,\ \ \ \frac{r_k^2(1+c)}{w_k}=q_k^2.
 \end{equation}
 The three equations in \eqref{ou-eq1-3} imply
 $$
(1+c_k)=\frac{p_k^2}{r_k^2},\ w_k= \frac{p_k^2}{q_k^2},\  h_k=\frac{1}{p_k^2q_k^2},
$$
and then \eqref{ps-ou1} follows from \eqref{OU-c}. Note that a particular
case of \eqref{ps-ou1} is $q_k=1/r_k$, $p_k^{-2}+1=r_k^{-2}$,
which is consistent with Theorem \ref{th:main-ps1} if we require the
range of $\bD_u$ to be in the domain of $\bdel_u$.

  \end{proof}

\begin{example}{\rm
Let  $\cU=X=\bR$. Then $\ba=n\in \{0,1,2,\ldots\}$,
$$
\xi_{\ba}:=\xi_{(n)}=\frac{\Hep_n(\xi)}{\sqrt{n}}, \ \xi:=\xi_{(1)},
$$
$$
u=\sum_{n\geq 0}u_n\xi_{(n)},\ v=\sum_{n\geq 0}u_n\xi_{(n)},\
f=\sum_{n\geq 0}f_n\xi_{(n)},\ u_n,v_n,f_n\in \bR.
$$
To begin, take $u=\xi$. Then
\begin{equation*}
 \begin{split}
 \bD_{\xi}(v)=\sum_{n\geq 1}&\sqrt{n}v_n\xi_{(n-1)},\
 \bdel_{\xi}(f)=\sum_{n\geq 0} \sqrt{n+1} f_{n+1} \xi_{(n)},\\
 &\cL_u(v)=\sum_{n\geq 1} n v_n \xi_{(n)}.
 \end{split}
 \end{equation*}

Next, let us illustrate the results of Theorems
\ref{th:main-ps1} and \ref{OU-PSp}. Let  $p,q,r$ be positive real numbers such that $$
\frac{1}{p}+\frac{1}{q}=\frac{1}{r},
$$
for example, $p=q=1,\ r=1/2$.
By Theorem  \ref{th:main-ps1}, if
$$
\sum_{n\geq 0} p^nu_n^2<\infty,\ \sum_{n\geq 0} \frac{v_n^2}{r^n} <\infty,
$$
then
$$
\sum_{n\geq 1} \frac{\big(\bD_u(v)\big)_n^2}{q^n}<\infty,
$$
and if
$$
\sum_{n\geq 0} p^nu_n^2<\infty,\ \sum_{n\geq 0} q^n{f_n^2} <\infty,
$$
then
$$
\sum_{n\geq 1} r^n{\big(\bdel_u(f)\big)_n^2}<\infty.
$$
If  $p,q,r$ are positive real numbers such that
$$
\left(\frac{1}{r}-\frac{1}{p}\right)\left(q-\frac{1}{p}\right)=1
$$
(for example, $p=1,\ q=2,\ r=1/2$) and
$$
\sum_{n\geq 0} p^nu_n^2<\infty,\ \sum_{n\geq 0} q^n{v_n^2} <\infty,
$$
then, by Theorem \ref{OU-PSp},
$$
\sum_{n\geq 0} r^n\big(\cL_u(v)\big)_n^2<\infty.
$$}
\end{example}


\providecommand{\bysame}{\leavevmode\hbox to3em{\hrulefill}\thinspace}
\providecommand{\MR}{\relax\ifhmode\unskip\space\fi MR }
\providecommand{\MRhref}[2]{%
  \href{http://www.ams.org/mathscinet-getitem?mr=#1}{#2}
}
\providecommand{\href}[2]{#2}

\end{document}